\documentclass[intlimits]{cmiae}

\usepackage{cite}

\usepackage{amsfonts}

\newtheorem{lem}{Lemma}
\newtheorem{thm}{Theorem}
\newtheorem{cor}{Corollary}
\newtheorem{proposition}{Proposition}
\newtheorem{remark}{Remark}

\newcommand{\restr}[2]{\left. #1 \right|_{#2}}
\newcommand{\sgn}{\operatorname{sgn}\nolimits}
\newcommand{\hypoth}[2]{\medskip\noindent{#1}{\em #2}\medskip}
\newcommand{\ddef}[1]{{#1}}
\newcommand{\am}{\operatorname{am}\nolimits}
\newcommand{\pder}[2]{\frac{\partial \, #1}{\partial \, #2} }
\newcommand{\sn}{\operatorname{sn}\nolimits}
\newcommand{\cn}{\operatorname{cn}\nolimits}
\newcommand{\dn}{\operatorname{dn}\nolimits}
\newcommand{\E}{\operatorname{E}\nolimits}

\newcommand{\eq}[1]{$(\protect\ref{#1})$}
\newcommand{\be}[1]{\begin{equation}\label{#1}}
\newcommand{\ee}{\end{equation}}
\newcommand{\vect}[1]{\left( \begin{array}{c} #1 \end{array} \right)}
\newcommand{\MAX}{\operatorname{MAX}\nolimits}
\newcommand{\Exp}{\operatorname{Exp}\nolimits}
\newcommand{\Ker}{\operatorname{Ker}\nolimits}
\newcommand{\map}[3]{#1 \, : \, #2 \to #3}  
\newcommand{\const}{\operatorname{const}\nolimits}
\newcommand{\pdder}[2]{\frac{\partial^2 #1}{\partial \, {#2}^2} }
\def\Ho{$(\mathbf{H1})$} 
\def\Ht{$(\mathbf{H2})$} 
\def\Hth{$(\mathbf{H3})$} 
\def\Hf{$(\mathbf{H4})$} 
\def\tconj{t^1_{\operatorname{conj}}}
\def\tmax{t_{\MAX}^1}
\def\tmaxd{t_{\MAX}^2}
\def\a{\alpha}
\def\tcut{t_{\operatorname{cut}}}
\def\tmax{t_{\MAX}^1}
\def\then{\quad\Rightarrow\quad}
\def\lam{\lambda}
\def\R{{\mathbb R}}
\def\ds{\displaystyle}
\def\eps{\varepsilon}
\def\tnu{\widetilde{\nu}}
\def\tlam{\widetilde{\lam}}
\def\bu{\bar{u}}
\def\f{\varphi}
\def\vH{\vec H}
\def\N{{\mathbb N}}
\def\Z{{\mathbb Z}}
\def\F{\Phi}

\begin{document}
%\subjclass{517.91; 517.936; 517.937}

\title
{Conjugate points in nilpotent sub-Riemannian problem \\on the Engel group}

\author{A. A. Ardentov, Yu. L. Sachkov}
\address{Program Systems Institute of RAS,
Pereslavl-Zalessky,	 Russia}
\email{aaa@pereslavl.ru, sachkov@sys.botik.ru}

\begin{abstract}
The left--invariant sub-Riemannian problem on the Engel group is considered. This problem is very important as nilpotent approximation of nonholonomic systems in four--dimensional space with two--dimensional control, for instance of a system which describes movement of mobile trailer robot. We study the local optimality of extremal trajectories and estimate conjugate time in this article.
\end{abstract}

\maketitle

\section{Introduction}
This work deals with the nilpotent sub--Riemannian problem on the Engel group with growth vector $(2, 3, 4)$.
Four-dimensional optimal control problem with two-dimensional control is stated as follows:
\begin{align} 
&\dot{q} = \vect{\dot{x} \\ \dot{y} \\ \dot{z} \\ \dot{v}} = u_1 
\vect{1 \\ 0 \\ - \frac{y}{2} \\ 0} + u_2 \vect{0 \\ 1 \\ \frac{x}{2} 
\\ \frac {x^2+y^2}{2}}, \quad q \in \R^4, \quad u \in \R^2, \label{pr1}\\
& q(0) = q_0 = (x_0,y_0,z_0,v_0), \quad q(t_1) = q_1 = (x_1, y_1, z_1, 
v_1), \label{pr2} \\
& l = \int_0^{t_1} \sqrt{u_1^2 + u_2^2} \, d t \rightarrow \min.
\label{pr3}
\end{align}
Since the problem is invariant under left shifts on the Engel group, we can assume that the initial point is identity of the group $q_0 = (x_0, y_0, z_0, v_0) = (0, 0, 0, 0)$.

The paper continues the study of this problem started in the work~\cite{engel}. The main result of~\cite{engel} is upper bound of the cut time (i.~e., the time of loss of { \em global} optimality) along extremal trajectories of the problem. The aim of this paper is to investigate the first conjugate time (i.~e., the time of loss of {\em local} optimality) along the trajectories. We show that the function that gives the upper bound of the cut time provides the lower bound of the first conjugate time. In order to state this main result exactly, we recall necessary facts from the previous work~\cite{engel}.

Existence of optimal solutions of problem~\eq{pr1}--\eq{pr3} is implied by Filippov's theorem~\cite{notes}. By Cauchy–-Schwarz inequality, it follows that sub-Riemannian length minimization problem~\eq{pr3} is equivalent to action minimization problem:
\be{J}
\int_0^{t_1} \frac{u_1^2+u_2^2}{2} \, d t \rightarrow \min.
\ee
Pontryagin's maximum principle~\cite{PGBM, notes} was applied to the resulting optimal control problem~\eq{pr1}, \eq{pr2}, \eq{J}. Abnormal extremals were parameterized.

Denote vector fields at the controls in the right-hand side of system~\eq{pr1}:
\begin{align*}
&X_1 = (1, 0, - \frac{y}{2}, 0)^T,& X_2 =(0, 1, \frac{x}{2}, \frac 
{x^2+y^2}{2})^T,
\end{align*}
and the corresponding linear on fibers of the cotangent bundle $T^* M$ Hamiltonians $h_i(\lambda)= \langle\lambda, X_i(q)\rangle$, $\lambda \in T^* M $, $i=1,2$. Normal extremals satisfy the Hamiltonian system
\be{norm_ham}
\dot{\lambda}= \vec{H}(\lambda), \qquad \lambda \in T^* M,
\ee 
where $H = \frac{1}{2}\left(h_1^2+h_2^2\right)$. 

The normal Hamiltonian system~\eq{norm_ham} is given, in certain natural coordinates, as follows on a level surface $\left\{\lam \in T^* M \mid H=\frac{1}{2}\right\}$:
\begin{eqnarray}
&&\dot{\theta}= c, \qquad \dot{c}= - \alpha\, \sin \theta, \qquad \dot{\alpha}=0, \label{pend}\\ 
&&\dot{q}= \cos \theta \, X_1(q)+\sin\theta \, X_2 (q), \qquad q(0)=q_0. \nonumber
\end{eqnarray}
The family of all normal extremals is parameterized by points of the phase cylinder of pendulum
\begin{eqnarray*}
C = \left\{\lambda \in T_{q_0} ^* M \mid H(\lambda)= \frac{1}{2}\right\}= \left\{(\theta, c, \alpha  ) \mid \theta \in S^1, \ c, \alpha \in \R  \right\},
\end{eqnarray*}
and is given by the exponential mapping
\begin{eqnarray*}
&&\map{\Exp}{N=C \times \R_+}{M},\\
&&\Exp (\lambda, t) = q_t = (x_t, y_t, z_t, v_t).
\end{eqnarray*}
Energy integral of pendulum~\eq{pend} is expressed by $E=\frac{c^2}{2}-\alpha \cos \theta$. The cylinder $C$ has the following stratification corresponding to the particular type of trajectories of the pendulum:
\begin{align*}
&C=\cup_{i=1}^7 C_i,   \quad  C_i \cap C_j = \emptyset, \ i \neq j, 
\quad \lambda = (\theta,c,\alpha),\\
&C_1 = \{\lambda \in C \mid \alpha \neq 0, E\in(- |\alpha|, 
|\alpha|)\}, \\
&C_2 = \{\lambda \in C \mid \alpha \neq 0, E\in(|\alpha|,+\infty)\}, 
\\
&C_3 = \{\lambda \in C \mid \alpha \neq 0, E=|\alpha|, c \neq 0 \}, \\
&C_4 = \{\lambda \in C \mid \alpha \neq 0, E=-|\alpha|\}, \\
&C_5 = \{\lambda \in C \mid \alpha \neq 0, E=|\alpha|, c = 0\}, \\
&C_{6} = \{\lambda \in C \mid \alpha = 0, \ c \neq 0\}, \\
&C_7 = \{\lambda \in C \mid \alpha = c = 0\}. 
\end{align*}

Extremal trajectories were parameterized by elliptic Jacobi's functions for any $\lambda \in C$ in the paper~\cite{engel}. This parameterization was obtained in natural coordinates $(\varphi, k, \alpha)$, which rectify the equations of pendulum: $\dot{\varphi}=1$, $\dot{k}=0$, $\dot{\alpha}=0$.

Further, in the work~\cite{engel} discrete symmetries of the exponential mapping were described. The corresponding Maxwell sets were constructed. On this basis was obtained the main result of the paper~\cite{engel}, Theorem~\ref{th:tcut_bound}, which gives upper bound of the cut time along extremal curves
$$\tcut(\lambda) = \sup \{ t>0 \mid \mathrm {Exp} (\lambda, s) \text{ is optimal for } s \in [0,t]\}.$$
Define the following function $\map{\tmax}{C}{(0, +\infty]}$:
\begin{align*}
&\lambda \in C_1 \then \tmax = \min (2 p_z^1, 4 K)/\sigma, \\
&\lambda \in C_2 \then \tmax = 2 K k/\sigma, \\
&\lambda \in C_6 \then \tmax = \frac {2 \pi} {|c|}, \\
&\lambda \in C_3 \cup C_4 \cup C_5 \cup C_7 \then \tmax=+\infty.
\end{align*}
where $\sigma = \sqrt{|\alpha|}$; $\ds K(k)= \int_0^\frac{\pi}{2} \frac{dt}{\sqrt{1- k^2 \sin^2 t}}$;
$p^1_z(k)\in (K(k), 3K(k))$ is the first positive root of the function $f_z(p,k)=\dn p \,\sn p+ (p-2\E(p))\cn p$; $\dn p$, $\sn p$, $\cn p$ are Jacobi's functions~\cite{whit_vatson}; $\E(p)=\int_0^p \dn^2 t \,dt$.

\begin{thm}[\cite{engel}, Theorem 3]\label{th:tcut_bound}
For any $\lambda \in C$ 
\begin{align}
\tcut(\lambda) \leq \tmax(\lambda). \label{tcutbound}
\end{align}
\end{thm}

We study the local optimality of extremal trajectories and estimate conjugate time in this article. A point $q_t = \Exp(\lambda, t) $ is called {\em a conjugate point} for $q_0$ if $\nu = (\lambda, t)$ is a critical point of the exponential mapping and that is why $q_t$ is the corresponding critical value:
$$
\map{d_{\nu} \Exp}{T_{\nu}N}{T_{q_t}M} \text{ is degenerate},
$$
i.~e.,
$$ \frac{\partial (x, y, z, v)}{\partial (\theta, c, \alpha, t)} \left(\nu\right)=0.$$
Note that $t$ is called {\em a conjugate time} along extremal trajectory $q_s = \Exp(\lambda, s)$, $s\geq 0$.

Here and below we denote by $\ds\frac{\partial (x, y, z, v)}{\partial (\theta, c, \alpha, t)}$ the Jacobian of the exponential map
$$
\begin{array}{|c c c|}
\pder{x}{\theta} & \ldots & \pder{x}{t} \\
\vdots & \ddots & \vdots \\
\pder{v}{\theta} & \ldots & \pder{v}{t}
\end{array} \, .
$$

Due to the strong Legendre condition, for any normal extremal there exists a countable family of conjugate points. Besides, conjugate times are separated from each other (see Section~\ref{sec:conj_hom}). The first conjugate time along the trajectory $\Exp (\lambda, s)$ is denoted by
$$\tconj=\min \left\{t>0 \mid t \text{ is a conjugate time along } \Exp (\lambda, s), \ s \geq 0 \right\}. $$

The trajectory $\Exp (\lambda, s)$ loses local optimality at the moment $t=\tconj(\lambda)$ (see~\cite{notes}). Our main aim is to prove the following lower bound of the first conjugate time.
\begin{thm}\label{th:tconjmax}
For any $\lambda \in C$
\be{tconjmax}
\tconj(\lambda) \geq \tmax(\lambda).
\ee 
\end{thm}

In Sections~\ref{sec:C1}--\ref{sec:C47} we prove the inequality~\eq{tconjmax}, $\lambda \in C_i$ for all $i=1,\dots,7$ (see Theorems \ref{th:tconjmaxC1}, \ref{th:tconjmaxC2}, \ref{th:tconjmaxC3}, \ref{th:tconjmaxC47}).

\section{Conjugate time and symmetries of the exponential mapping}
Normal Hamiltonian system for the considered problem has the following symmetries (see~\cite{engel}):
reflection
\be{refl}
(\theta, c, \alpha, x, y, z, v, t) \mapsto (\theta - \pi, c, -\alpha, -x, -y, z, -v, t )
\ee
and dilations
\be{dil}
(\theta, c, \alpha, x, y, z, v, t) \mapsto  (\theta , \frac{c}{\sqrt{\gamma}}, \frac{\alpha}{\gamma}, \sqrt{\gamma}x, \sqrt{\gamma}y, \gamma z, \gamma^{\frac{3}{2}}v, \sqrt{\gamma} t), \quad \gamma>0.
\ee 
We consider the corresponding symmetries of the exponential mapping and their action on conjugate points.

\subsection{Reflection}
Define the action of reflection in preimage and image of the exponential mapping according to~\eq{refl}:
\begin{eqnarray*}
&&i: N\rightarrow N,\qquad  i(\nu)=i(\theta, c, \alpha, t)=\widetilde{\nu}= (\theta - \pi, c, -\alpha, t),\\
&&i: M\rightarrow M,\qquad   i(q)= i(x, y, z, v)=\widetilde{q}= (-x, -y, z, -v).
\end{eqnarray*}
Existence of symmetry~\eq{refl} of Hamiltonian system implies that the reflection $i$ is the symmetry of the exponential mapping: $\Exp\circ i = i\circ \Exp$. (It is easily shown that $i$ and the reflection $\eps^4$ coincide~\cite{engel}). Hence we obtain $d \Exp \circ d i = d i \circ d \Exp$. The reflection $i$ is non-degenerate ($\Ker \quad d i={0}$) and therefore $\nu = (\lambda, t)= (\theta, c, \alpha, t)$ is a critical point of $\Exp$ if and only if $\tnu= i(\nu)= (\tlam, t)= (\theta - \pi, c, -\alpha, t)$ is a critical point of $\Exp$. And so $\tconj (\tlam)= \tconj(\lambda)$. Using the definition of Maxwell time $\tmax$ (see \cite{engel}, p. 7.6.), we get easily similar equality $\tmax (\tlam)= \tmax (\lam)$. Therefore it is enough to prove the necessary inequality~\eq{tconjmax} can be proved only for $\alpha \geq 0$.

\subsection{Dilations}\label{subsec:dilations}
According to formula \eq{dil} define the action of dilations in preimage and image of the exponential mapping:
\begin{align*}
&\F_\gamma: N \rightarrow N,\quad \F_\gamma(\nu) &&=\F_\gamma(\lambda, t)= \F_\gamma(\theta, c, \alpha, t)=(\tilde{\lambda}, \tilde{t}) = \left( \theta, \frac{c}{\sqrt{\gamma}}, \frac{\alpha}{\gamma}, \sqrt{\gamma}, \sqrt{\gamma}t\right),\\
&\F_\gamma: M \rightarrow M, \quad \F_\gamma(q)&&= \F_\gamma (x, y, z, v)= \tilde{q}= \left(\sqrt{\gamma}x,\sqrt{\gamma}y, \gamma z, \gamma^{\frac{3}{2}} v  \right),
\quad \gamma>0.
\end{align*}
These formulas define the action of multiplicative Lie group $\R_+$ in $N$ and $M$, s.~t.
$$\Exp\circ \F_\gamma = \F_\gamma \circ \Exp \qquad \forall\gamma > 0.$$
Thus, there is a one-dimensional symmetry group $G=\left\{\F_\gamma| \gamma>0\right\}$ of the exponential map.

It is easy to see that the symmetries preserve the sets of critical points and critical values of the exponential mapping.

\begin{lem} 
\label{lem:tconjfactor}
\begin{itemize}
	\item[$1)$] 
If $q\in M$ is the critical value of $\Exp$ corresponding to a critical point $\nu \in N$ then $\F_\gamma(q)$ is also the critical value of $\Exp$ corresponding to the critical point $\F_\gamma(\nu)$ for any $\gamma >0$. 
	\item[$2)$]
Let $\gamma >0$, $\lambda = (\theta, c , \alpha )$, $\tlam = \left(\theta, \frac{c}{\sqrt{\gamma}}, \frac{\alpha}{\gamma} \right)\in C$. Then $\tconj(\lambda)= \frac{1}{\sqrt{\gamma}} \tconj (\tlam)$.	
\end{itemize}
\end{lem}
\begin{proof}
1) follows from the equality $d \Exp \circ d \F_\gamma = d \F_\gamma \circ d \Exp$.

2) follows from 1).
\end{proof}

Let $\alpha > 0$. Suppose $\gamma = \alpha$; then from Lemma~\ref{lem:tconjfactor}, we get the following:
$$
\tconj (\theta, c, \alpha) = \frac{1}{\sqrt{\alpha}} \tconj \left(\theta, \frac{c}{\sqrt{\alpha}}, 1 \right).$$
From the definition of Maxwell time $\tmax$, a similar equation follows:
$$
\tmax (\theta, c, \alpha) = \frac{1}{\sqrt{\alpha}}\tmax \left(\theta, \frac{c}{\sqrt{\alpha}}, 1 \right).$$
Therefore, it is sufficient to prove the required inequality~\eq{tconjmax} in two cases: for $\alpha =1$ and $\alpha = 0$.

\subsection{Transformation of Jacobian of the exponential mapping}
\label{subsec:J_transform}
For a fixed $\lambda = (\theta, c, \alpha)$, conjugate times are roots $t>0$ of the Jacobian $\ds\frac{\partial (x, y, z, v)}{\partial (\theta, c, \alpha, t)}$. First, we transform this Jacobian by using the symmetry group $G=\left\{\F_\gamma| \gamma>0\right\}$. The coordinate expression of the equation $\Exp\circ \F_\gamma(\lambda, t) = \F_\gamma\circ \Exp(\lambda, t)$ is
$$\Exp \left(\theta, \frac{c}{\sqrt{\gamma}}, \frac{\alpha}{\gamma}, \sqrt{\gamma}t \right) = \left(\sqrt{\gamma}x, \sqrt{\gamma}y, \gamma z, \gamma^{\frac{3}{2}}v \right).$$
Differentiating this equation w.r.t. $\gamma$ for $\gamma=\alpha=1$, we get 
$$-\frac{c}{2}\frac{\partial q}{\partial c} - \frac{\partial q}{\partial \alpha}+ \frac{t}{2}\frac{\partial q}{\partial t}= \left(\frac{x}{2}, \frac{y}{2}, z, \frac{3}{2} v \right)=: L. $$
Therefore, when $\alpha = 1$
\begin{eqnarray*}
\frac{\partial (x, y, z, v)}{\partial (\theta, c, \alpha, t)}  
&& = \det \left(\frac{\partial q}{\partial \theta}, \frac{\partial q}{\partial c}, \frac{\partial q}{\partial \alpha}, \frac{ \partial q }{\partial t}\right)=
\det \left(\frac{\partial q}{\partial \theta}, \frac{\partial q}{\partial c}, -\frac{c}{2}\frac{\partial q }{\partial c}+ \frac{t}{2}\frac{\partial q}{ \partial t}-L, \frac{\partial q}{\partial t}\right) = \\
&& = \det \left(\frac{\partial q}{\partial \theta}, \frac{\partial q}{\partial c}, \frac{ \partial q }{\partial t}, L \right)=\frac{1}{2} \begin{array}{|c c c c|}
\frac{\partial x}{\partial \theta} & \frac{\partial x}{\partial c} & \frac{\partial x}{\partial t} & x \\
\frac{\partial y}{\partial \theta} & \frac{\partial y}{\partial c} &  \frac{\partial y}{\partial t} & y \\
\frac{\partial z}{\partial \theta} & \frac{\partial z}{\partial c} & \frac{\partial z}{\partial t} & 2z \\
\frac{\partial v}{\partial \theta} & \frac{\partial v}{\partial c} & \frac{\partial v}{\partial t} & 3v
\end{array}\,.
\end{eqnarray*}

\section{Conjugate points and homotopy}\label{sec:conj_hom}

In this section we recall some necessary facts from the theory of conjugate points in optimal control problems. For details see, e.g., \cite{notes, cime, sar}.

Consider an optimal control problem of the form
\begin{align}
&\dot q = f(q,u), \qquad q \in M, \quad u \in U \subset \R^m, \label{sys} \\
&q(0) = q_0, \qquad q(t_1) = q_1, \qquad t_1 \text{ fixed}, \label{bound1} \\
&J = \int_0^{t_1} \f(q(t),u(t)) \, dt \to \min, \label{J1}
\end{align}
where $M$ is a  finite-dimensional analytic manifold,  $f(q,u)$ and $\f(q,u)$ are respectively analytic in $(q,u)$  families of vector fields   and   functions  on $M$ depending on the control parameter $u \in U$, and $U$ an open subset of $\R^m$. Admissible controls are $u(\cdot) \in L_{\infty}([0, t_1],U)$, and admissible trajectories $q(\cdot)$ are Lipschitzian.
Let 
$$
h_u(\lam) = \langle \lam, f(q,u)\rangle - \f(q,u), 
\qquad \lam \in T^*M, \quad q = \pi(\lam) \in M, \quad u \in U,
$$
be the \ddef{normal Hamiltonian of PMP} for problem~\eq{sys}--\eq{J1}. Fix a triple $(\widetilde{u}(t), \lam_t, q(t))$ consisting of a normal extremal control $\widetilde{u}(t)$, the corresponding extremal $\lam_t$, and the extremal trajectory $q(t)$ for the problem~\eq{sys}--\eq{J1}.

Let the following hypotheses hold:

\hypoth{\Ho}{For all $\lam \in T^*M$ and  $u \in U$, the quadratic form $\ds \pdder{h_u}{u}(\lam)$ is negative definite.}

\hypoth{\Ht}
{For any $\lam \in T^* M$, the function $u \mapsto h_u(\lam)$, $u \in U$, has a maximum point $\bu(\lam) \in U$:
$$
h_{\bu(\lam)}(\lam) = \max_{u \in U} h_u(\lam), \qquad \lam \in T^*M.
$$}%

\hypoth{\Hth}
{The extremal control $\widetilde{u}(\cdot)$ is a corank one critical point of the endpoint mapping.}%

\hypoth{\Hf}
{All trajectories of the Hamiltonian vector field $\vH(\lam)$, $\lam \in T^*M$, are continued for $t \in [0, +\infty)$. 
}

An instant $t_* > 0$ is called a \ddef{conjugate time} (for the initial instant $t = 0$) along the extremal $\lam_t$ if the restriction of the second variation of the endpoint mapping to the kernel of its first variation   is degenerate, see~\cite{notes} for details. In this case the point $q(t_*) = \pi(\lam_{t_*})$ is called \ddef{conjugate} for the initial point  $q_0$ along the extremal trajectory $q(\cdot)$.

Under hypotheses \Ho--\Hf, we have the following:

\begin{enumerate}
\item
Normal extremal trajectories lose their local optimality (both strong and weak) at the first conjugate point, see~\cite{notes}.
\item
An instant $t > 0$ is a conjugate time iff the exponential mapping $\Exp_t = \pi \circ e^{t \vH}$ is degenerate, see~\cite{cime}.
\item
Along each normal extremal trajectory, conjugate times are isolated one from another, see~\cite{sar}.
\end{enumerate}

We will apply the following statement for the proof of absence of conjugate points via homotopy.

\begin{thm}[Corollary 2.2~\cite{el_conj}]\label{th:conj_hom}
Let $(u^s(t), \lam^s_t)$, $t \in [0, + \infty)$, $s \in [0, 1]$, be a continuous in parameter~$s$ family of normal extremal pairs in the optimal control problem~\eq{sys}--\eq{J1} satisfying hypotheses \Ho--\Hf.

Let $s \mapsto t_1^s$ be a continuous function, $s \in [0, 1]$, $t_1^s \in (0, + \infty)$.
Assume that for any $s \in [0, 1]$ the instant $t = t_1^s$ is not a conjugate time along the extremal $\lam_t^s$.

If the extremal trajectory $q^0(t) = \pi(\lam_t^0)$, $t \in (0, t_1^0]$, does not contain conjugate points, then the extremal trajectory $q^1(t) = \pi(\lam_t^1)$, $t \in (0, t_1^1]$, also does not contain conjugate points.
\end{thm}

One easily checks that the sub-Riemannian problem~\eq{pr1}, \eq{pr2}, \eq{J} satisfies all hypotheses \Ho--\Hf, so the results cited in this section are applicable to this problem.

\section{Estimate of conjugate time for $\lambda \in C_1$}
\label{sec:C1}
\subsection{Evaluation of Jacobian }
We use the elliptic coordinates $(\varphi, k, \alpha)$ in $C_1$, see~\cite{engel}. For a fixed $\lambda = (\theta, c, \alpha) \in C_1$, conjugate times are roots $t>0$ of the Jacobian
$$J= \frac{\partial(x, y,z,v)}{\partial(t,\varphi, k, \alpha)}.$$
We transform this Jacobian in the same way as the determinant $\ds\frac{\partial(x, y,z,v)}{\partial(\theta,c, \alpha, t)}$ in Subsection~\ref{subsec:J_transform}, then get
$$J=-\frac{1}{2} \, \begin{array}{|c c c c|}
\frac{\partial x}{\partial t} & \frac{\partial x}{\partial \varphi} & \frac{\partial x}{\partial k} & x \\
\frac{\partial y}{\partial t} & \frac{\partial y}{\partial \varphi} &  \frac{\partial y}{\partial k} & y \\
\frac{\partial z}{\partial t} & \frac{\partial z}{\partial \varphi} & \frac{\partial z}{\partial k} & 2z \\
\frac{\partial v}{\partial t} & \frac{\partial v}{\partial \varphi} & \frac{\partial v}{\partial k} & 3v
\end{array}\, .$$
(Here and below we assume $\alpha = 1$ according to Subsec.~\ref{subsec:dilations}.)
Explicit calculation of the function with the use of the parameterization of the exponential mapping obtained in~\cite{engel} gives the following expression of the determinant:
\begin{eqnarray}
&&J=R \cdot J_1, \nonumber\\
&&R=-\frac{32}{k(1-k^2)(1-k^2\sin^2 u_1 \sin^2 u_2)^2}\neq 0, \nonumber\\
&&J_1=d_0+d_2 \sin^2 u_2 +d_4 \sin^4 u_2, \nonumber\\
&& d_i = d_i(u_1, k), \quad i=0,2,4, \nonumber\\
&&u_1= \am(p, k),\quad u_2= \am (\tau,k), \label{uiptau}\\
&& p= \frac{t}{2}, \quad \tau=\varphi+\frac{t}{2}, \label{ptautphi}\\
&&d_0=a_1\cdot\sin u_1,\label{d0}\\
&&d_4=k^2 a_2\cdot f_{zu},\label{d4}\\
&&d_0+d_2 = -a_2\cdot f_{zu},\label{d0+d2}\\
\end{eqnarray}
\begin{multline*}
a_1= \frac{1}{2}
\left[4(1-k^2) \cos u_1 \left( 1- 2k^2 \sin^2 u_1 \right) \sqrt{1-k^2 \sin^2 u_1}F^2(u_1)+ %\right.\\
 +4k^2 \cos u_1 \sin^2 u_1 \left(1- \right. \right. \\
 \left. -k^2 \sin^2 u_1\right)^{\frac{3}{2}}+ 4 \sin u_1\left(1-k^2\sin^2 u_1\right)E^3(u_1)-2 \left(1-k^2\right) \sin u_1 \left(1-k^2 \sin^2 u_1\right) F^3(u_1)+\\
 +2F(u_1)\left(\sin u_1-2k^2\left(3-2k^2\right)\sin^3 u_1+k^4\left(5-4k^2\right)\sin^5 u_1\right)+
 +E^2(u_1) ( 2 \left(4k^2-5\right) \left(1- \right.\\
 \left. -k^2\sin^2 u_1\right) \sin u_1 F(u_1)+ 6\cos u_1 \left( 1-2k^2\sin^2 u_1 \right) \sqrt{1-k^2 \sin^2 u_1})+ E(u_1)\left(2\left(4k^2-\right. \right. \\
 \left. - 5\right)\cos u_1 \left(1-2 k^2 \sin^2 u_1\right) \sqrt{1-k^2 \sin^2 u_1} F(u_1)+ 8\left(1-k^2\right)\left(1-k^2 \sin^2 u_1\right)\sin u_1  F^2(u_1)-\\
\left.\left. -2\left(1+k^2+ 3k^2 \cos(2u_1)\right)\sin u_1 \left(1-k^2 \sin^2 u_1\right)
\vphantom{\sqrt{1-k^2 \sin^2 u_1}} \right)\right],
\end{multline*}
\begin{multline*}
a_2= -\cos u_1\left(\left(E(u_1)-F(u_1)\right)^2+k^2 F(u_1) \left(2E(u_1)-F(u_1)\right)\right)-\\
-\sin u_1 \sqrt{1-k^2 \sin^2 u_1}\left( E(u_1)-\left(1-k^2\right) F(u_1)\right),
\end{multline*}
\begin{align*}
f_{zu} = \sin u_1 \sqrt{1-k^2 \sin^2 u_1}+\left(F(u_1)-2E(u_1)\right)\cos u_1.
\end{align*}

Denote
\be{xu2}
x= \sin^2 u_2.
\ee

\subsection{Conjugate points as $k \to 0$}

We show that extremals corresponding to sufficiently small values of the parameter $k$ have no conjugate points for $t < \tmax(\lam)$.

The function $J_1$ has the following asymptotics as $k \to 0$:
	\begin{align*}
&J_1(u_1, x, k) = k^2 J_1^0 (u_1,x) + o(k^2), \qquad x = \sin^2 u_2,\\
&J_1^0 (u_1, x) = d_0^0(u_1) + d_2^0 (u_1) x, \\
&d_0^0 (u_1) = \frac{1}{2} \sin u_1 \big(2 u_1^3 \sin u_1 + 3 u_1^2 \cos u_1 + u_1 \sin^3 u_1 - 6 u_1 \sin u_1 + 3 \cos u_1 \sin^2 u_1 \big), \\
&d_2^0 (u_1) = - u_1 \sin^4 u_1 - 2 \cos u_1 \sin^3 u_1 + 3 u_1 \sin^2 u_1 - u_1^3. 
\end{align*}

\subsubsection{Auxiliary lemmas}

We use the following statement to obtain the necessary bounds for functions.
\begin{lem}\label{lem:l1}
Let real analytic functions $f(u), g(u)$ satisfy on $(0,u_0)$ the conditions:
\begin{align}
&f(u) \not \equiv 0, \qquad g(u)>0, \qquad \bigg(\frac{f(u)}{g(u)}\bigg)'\geq 0, \label{cond1}\\
&\lim_{u \to 0} \frac{f(u)}{g(u)} = 0. \label{cond2}
\end{align}
Then $f(u)>0$ for $u \in (0, u_0)$.
\end{lem} 

If functions $f$ and $g$ satisfy conditions~\eq{cond1},~\eq{cond2}, then we say that $g$ is a comparison function for $f$ on the interval $(0,u_0)$.

\begin{proof}
The function $\ds\left(\frac{f}{g}\right)'$ is real analytic, thus it either has isolated zeros or is identically zero. It is not hard to prove that the second case is impossible: if $\ds\left(\frac{f}{g}\right)' \equiv 0$ then $\ds\frac{f}{g} \equiv \const$, hence $\ds\frac{f}{g} \equiv 0$ (because $\ds\lim_{u \to 0} \frac{f(u)}{g(u)} = 0$), whence $f \equiv 0$, this contradiction proves the case.

So the function $\ds\left(\frac{f}{g}\right)' \geq 0$ has isolated zeros therefore $\ds\frac{f}{g}$ strictly increases for $u \in (0, u_0)$. The inequality $\ds\frac{f(u)}{g(u)}>0$ follows from the equality~\eq{cond2}, so the inequality $f(u)>0$ for $u \in (0,u_0)$ follows from $g(u)>0$.
\end{proof}

We use the following statement to estimate a quadratic polynomial.

\begin{lem}\label{lem:l4}
If $f(0,y)>0, f(1,y)>0$ and $a(y) \leq 0$ for the function $f(x,y) = a(y) x^2 + b(y) x + c(y)$ with $y \in (0, y_0)$, then $f(x,y) > 0$ for $y \in (0, y_0), x \in [0,1]$.
\end{lem} 

\begin{proof}
We obviously have the inequality $f(x,y)>0$ for $x=0$ and $x=1$, $y \in (0,y_0)$. Since $a(y) \leq 0$, it follows that $f(x,y)$ is convex w.~r.~t. the variable $x$ (possibly not strictly). Consequently we get $f(x,y)>0$ for $y \in (0, y_0), x \in [0,1]$.
\end{proof}

In the following three lemmas we analyze the sign of the function $J^0_1(u_1, x)$, which is dominant term of asymptotics for the function $J_1(u_1, x, k)$ as $k \rightarrow 0$.
\begin{lem}\label{lem:l2}
The function $d_0^0(u_1) = \frac{1}{2} \sin u_1 \big(2 u_1 \sin u_1 (u_1^2 -3) + 3 \cos u_1 (u_1^2+\sin^2 u_1) + u_1 \sin^3 u_1 \big) < 0$ for $u_1 \in (0, \pi)$.
\end{lem} 

\begin{proof}
We show that the function $g(u_1) = \sin u_1 (\sin u_1 - u_1 \cos u_1)$ is a comparison function for $-d_0^0 (u_1)$ for $u_1 \in (0, \pi)$.

The inequality $d_0^0 (u_1) \not \equiv 0$ follows from the expansion $d_0^0 (u_1) = - \frac{4}{4725} u^{11} + o(u^{11})$.

If $u_1 \in (0, \pi)$, then $\sin u_1 >0$. Further, $\phi (u_1) = \sin u_1 - u_1 \cos u_1 > 0$ for $u_1 \in (0, \pi)$, since $\phi (0) = 0, \phi'(u_1) = u_1 \sin u_1 >0$. Therefore $g(u_1)>0$ for $u_1 \in (0, \pi)$.

Finally we get the equalities
$$\ds\left(\frac{-d_0^0(u_1)}{g(u_1)}\right)' = \frac{(2(-1+u_1^2+\cos(2 u_1)) + u_1 \sin(2 u_1))^2}{4 (\sin u_1 - u_1 \cos u_1)^2}$$ and $\ds\frac{-d_0^0(u_1)}{g(u_1)}=\frac{4}{1575}u_1^7+o(u_1^7)$.

So $g(u_1)$ is a comparison function for $-d_0^0(u_1)$ thus it follows from Lemma~\ref{lem:l1} that $d_0^0(u_1)<0$ for $u_1 \in (0, \pi)$.
\end{proof}

\begin{lem}\label{lem:l3}
If $u_1 \in (0, \pi)$, then
\begin{multline*}
d_0^0(u_1)+d_2^0(u_1) = \frac{1}{2} \big(-2 u_1^3 -\sin u_1 (-2 u_1^3 \sin u_1 + u_1 \sin^3 u_1 + \cos u_1 (-3 u_1^2 + \sin^2 u_1) )\big) < 0.
\end{multline*}
\end{lem} 

\begin{proof}
We check that the function $g(u_1) = \sin u_1 (\sin u_1 - u_1 \cos u_1)$ is a comparison function for $-(d_0^0 (u_1)+d_2^0 (u_1))$ for $u_1 \in (0, \pi)$.

The inequality $d_0^0 (u_1)+d_2^0 (u_1) \not \equiv 0$ follows from the expansion $d_0^0 (u_1)+d_2^0 (u_1) = - \frac{4}{135} u^9 + o(u^9)$.

Note that $g(u_1)>0$ for $u_1 \in (0, \pi)$ (see the proof of Lemma~\ref{lem:l2}).

Also, there hold the equalities $$\left(\frac{-(d_0^0(u_1)+d_2^0(u_1))}{g(u_1)}\right)' = \frac{(-2 u_1 + \sin(2 u_1))^2}{4 \sin^2 u_1}$$ and $$\frac{-(d_0^0(u_1)+d_2^0(u_1))}{g(u_1)}=\frac{4}{45}u_1^5+o(u_1^5).$$

Finally, $g(u_1)$ is a comparison function for $-(d_0^0(u_1)+d_2^0(u_1))$ therefore it follows from Lemma~\ref{lem:l1} that $d_0^0(u_1)+d_2^0(u_1)<0$ for $u_1 \in (0, \pi)$.
\end{proof}

\begin{lem}\label{lem:J10<0}
For any $u_1 \in (0, \pi)$, $x \in [0,1]$, we have $J^0_1(u_1, x)<0$
\end{lem}
\begin{proof}
If $u_1 \in (0, \pi)$, then $J_1^0 (u_1,0) < 0$ (see Lemma~\ref{lem:l2}), $J_1^0 (u_1,1)<0$ (see Lemma~\ref{lem:l3}), and so it follows from Lemma~\ref{lem:l4} that $J_1^0 (u_1, x) <0$ for $u_1 \in (0, \pi), x \in [0,1]$. 
\end{proof}

\subsubsection{Estimate of conjugate time as $k \to 0$}
\begin{proposition} \label{prop:k=0u}
There exists $\bar{k}\in (0,1)$, s.~t. for any $k \in \left(0, \bar{k}\right)$, $u_1 \in (0, \pi )$, $x \in [0, 1]$ we have $J_1(u_1,x,k)<0$.
\end{proposition}
\begin{proof}
Assume the converse. Then there exist sequences $\{k_n \}$, $\left\{u^n_1\right\}$, $\left\{x_n \right\}$, $n \in \N$, s.~t. $k_n \in (0,1)$, $k_n \rightarrow 0$, $u^n_1 \in (0, \pi)$, $x_n \in [0,1]$, and  $J_1\left(u^n_1, x_n, k_n\right)\geq 0$ for all $n \in \N$. By passing to subsequences, we can assume that $u^n_1 \rightarrow \hat{u}_1 \in [0, \pi]$, $x_n \rightarrow \hat{x} \in [0, 1]$.

1)
Let $\hat{u}_1 \in (0,\pi)$. 
From Lemma~\ref{lem:J10<0}, we get $J^0_1 (u_1, x)<0$ for all $u_1 \in (0, \pi)$, $x \in [0,1]$. Thus $J_1 (u^n_1, x_n, k_n)= k^2_n \left(J^0_1\left(u^n_1, x_n\right)+o(1)\right)<0$ for large values of $n$, a contradiction.

2) 
Let $\hat{u}_1=0$. As $k^2+u^2_1 \rightarrow 0$ we have
	\begin{eqnarray*}
	&&d_0= -\frac{4}{4725}k^2 u^{11}_1 +o(k^2 u^{11}_1),\\
	&&d_2= - \frac{4}{135}k^2 u^9_1 +o(k^2 u^9_1),\\
	&&d_4= \frac{4}{135} k^4 u^9_1 +o(k^4 u^9_1).
   \end{eqnarray*}

2.1) 
If $\hat{x} \neq 0$, then $J_1 = - \frac{4}{135}k^2 u^9_1  x +o(k^2 u^9_1)$. Therefore $J_1(u^n_1, x_n, k_n)<0$ for large values of $n$, a contradiction.

2.2) 
If $\hat{x}=0$, then
	    $$J_1 = -\frac{4}{4725} k^2 u^{11}_1+ o(k^2 u^{11}_1)-  \frac{4}{135}k^2 u^9_1 x +o(k^2 u^9_1 x)$$ 
	    and $J_1(u^n_1, x_n, k_n)<0$ as $n \rightarrow \infty$, a contradiction.

3)
Let $\hat{u}_1=\pi$. As $k^2+(\pi-u_1)^2 \rightarrow 0$ we get 
   \begin{eqnarray*}
   &&d_0= - \frac{3}{2}\pi^2 k^2 (\pi - u_1)+ o\left(k^2(\pi -u_1)\right),\\
   &&d_2= -\pi^3 k^2 + o(k^2), \\
   &&d_4= \pi^3 k^4+o(k^4).   
    \end{eqnarray*}

3.1)
If $\hat{x}\neq0$, then $J_1=k^2\left(-\pi^3x +o(1)\right)$. Whence $J_1\left(u^n_1, x_n, k_n\right)<0$ as $n\rightarrow \infty$, a contradiction.

3.2)
If $\hat{x }=0$, then
	   $$J_1 = - \frac{3}{2}\pi^2 k^2 (\pi - u_1) + o \left(k^2 (\pi - u_1) \right)- \pi^3 k^2 x +o(k^2x). $$
	   Hence $J_1(u^n_1, x_n, k_n)<0$ as $n\rightarrow \infty $. The contradiction completes the proof.
\end{proof}

Going back from the variables $(u_1, x, k)$ to $(t, \varphi, k)$ by formulas~\eq{xu2}, \eq{uiptau}, \eq{ptautphi}, we get the following statement from Proposition~\ref{prop:k=0u}.

\begin{cor} \label{cor:k=0t}
There exists $\bar{k }\in (0,1)$, s.~t. for any $k \in (0, \bar{k })$, $\varphi \in \R$, an arc of the extremal curve $\Exp(\lambda,t)$, $\lambda= (\varphi, k, \alpha)$, $t \in (0, \tmax(\lambda))$, does not contain conjugate points.
\end{cor}

\subsection{Conjugate points at $t=\tmax$}
In this subsection we find conditions, under which Maxwell time $t=\tmax$ is a conjugate time. Let us recall that $\tmax(\lambda)= \min(2p^1_z(k), 4K(k))$ for $\lam \in C_1$, $\alpha = 1$, where $p=p^1_z(k)\in (K, 3K)$ is the first positive root of the function $f_z(p, k)= \dn p \sn p + (p -2\E(p)) \cn p$ (see~\cite{engel}).

It is shown in~\cite{max3} that 
\begin{align*}
&k \in (0, k_0) \then p^1_z(k)\in (3K, 2K),\\
&k=k_0 \then p^1_z(k)=2K,\\
&k\in (k_0, 1) \then p^1_z(k) \in (K,2K), 
\end{align*}
where $k_0 \approx 0.9$ is the unique root of the equation $2E(k) - K(k)= 0$. Therefore
$$
\tmax(\lambda)= \begin {cases} 4K(k) \text{ for } k \in \left(0, k_0\right],\\
2p^1_z (k)\text{ for } k \in \left[k_0, 1\right). \end{cases}$$

Changing the variable $t$ by $u_1 = \am\left(\frac{t}{2}, k\right)$, we get
$$
u_{\MAX}^1(k)= \begin {cases}\pi \text{ for } k \in \left(0, k_0\right],\\
u^1_z(k) \text{ for } k \in \left[k_0, 1\right), \end{cases}
$$
where $u_1 = u_z^1(k) = \am (p^1_z(k), k) \in \left(\frac{\pi}{2}, \frac{3\pi}{2}\right)$ is the first positive root of the function $f_{zu}(u_1, k) = f_z(\am u_1, k)$.

\begin{lem}\label{lem:2.8}
The function $f_3(k) = E(k) + 2 (k^2-1) E(k) K(k) - (k^2-1) K(k) > 0$ on the interval $k \in (0, 1)$.
\end{lem} 

\begin{proof}
Let us prove that the function $g(k) = 1-k^2$ is a comparison function for $f_3(k)$ on the interval $k \in (0, 1)$.

The inequality $f_3(k) \not \equiv 0$ follows from the expansion $f_3(k) = \frac{\pi^2}{4} k^2 + o(k^2)$.

Notice that $g(k)>0$ for $k \in (0,1)$.

Also, we have the equalities $\ds\left(\frac{f_3(k)}{g(k)}\right)' = \frac{2 k E^2(k)}{(k^2-1)^2}$ and $\ds\frac{f_3(k)}{g(k)}=\frac{\pi^2}{4} k^2 + o(k^2)$.

Finally, $g(k)$ is a comparison function for $f_3(k)$, hence it follows from Lemma~\ref{lem:l1} that $f_3(k)>0$ for $k \in (0, 1)$.
\end{proof}

\begin{lem} \label{lem:u1pi}
\begin{enumerate}
	\item[$1)$] Let $u_1= \pi$, $x \in \left(0,1\right]$; then $\sgn J_1 = -\sgn f_{zu}(\pi, k)= -\sgn(2E(k)- K(k))$, i.~e., $J_1<0$ for $k \in (0, k_0)$, $J_1=0$ if $k=k_0$, and $J_1>0$ for $k \in (k_0,1)$.
	\item[$2)$] If $u_1=\pi$, $x=0$, then $J_1=0$.
\end{enumerate}
\end{lem}
\begin{proof}
From a direct calculation it follows that
\begin{align}
&J_1(\pi, x, k) = -4 x (1- k^2 x) f_z (\pi, k) f_3 (k), \label{J1pix}\\
&f_z(\pi, k) = 2 E (k) - K(k), \nonumber\\
&f_3(k) = E(k) + 2 (k^2-1) E(k) K(k) - (k^2-1) K(k). \nonumber
\end{align}
Now the statement of item 1) of this lemma follows from Lemma~\ref{lem:2.8}  ($f_3(k)>0$ for all $k \in (0,1)$) and the distribution of signs of the function $2E(k)-K(k)$ \cite{max3} (this function is positive for $k < k_0$, equals zero if $k = k_0$ and is negative for $k > k_0$). 

The statement of item 2) follows from formula~\eq{J1pix}.
\end{proof}

\begin{lem} \label{lem:a1u1z}
Let $k \in (0,1)$, $k \neq k_0$, $u_1= u_z^1(k)$. Then $a_1(u_1, k)<0$.
\end{lem}
\begin{proof}
From a direct calculation it follows that if $u_2= u^1_z(k)$, then 
\begin{eqnarray*}
&& a_1 = \left( e_0 + e_1 F(u_1) + e_2 F^2 (u_1)\right)/ \left(4 \cos^3 u_1 \right),\\
&& e_0 = \cos^2 u_1 \sqrt{1- k^2 \sin ^2 u_1}\left(1-k^2 \left(1- \cos^4 u_1\right)\right),\\
&& e_1 = -2k^2 \cos u_1 \sin u_1 (1-\left(4-5\cos^2 u_1+ \cos^4 u_1\right) + \left(5-\cos^2 u_1\right) k^4 \sin^4 u_1 - 2k^6 \sin^6 u_1),\\
&& e_2 = \sin^2 u_1 \left(1- k^2 \sin^2 u_1\right)^{\frac{3}{2}} \left(1-k^2\left(1-\cos^4 u_1\right)\right).
\end{eqnarray*}
To estimate a sign of the function $a_1$ notice first that $u_1 = u^1_z \in \left(\frac{\pi}{2}, \frac{3 \pi}{2}\right)$, therefore $\cos^3 u_1 <0$. Further, we analyze a sign of the quadratic trinomial $h(z) = e_0+e_1z + e_2z^2$. Its discriminant is equal to 
$$D = e_1^2 - 4e_0e_2 = -16 k^2 \sin ^4 u_1 \cos^4 u_1 \left(1 - k^2 \sin ^2 u_1 \right)^4,$$ therefore $D<0$ for $k \neq k_0$. We have $h(z)>0$ for $k \neq k_0$ because $e_0 > 0$,  therefore $a_1<0$.
\end{proof}

\begin{lem} \label{lem:u1u1maxC1}
\begin{enumerate}
	\item[$1)$] 
Let $k \in (0, 1)$, $u_1=u^1_z(k)$, $x\in \left[0,1\right)$; then $\sgn J_1 = \sgn f_{zu}(\pi, k)= \sgn (2E(k)-K(k))$, i.~e., $J_1>0$ for $k \in (0, k_0)$, $J_1=0$ if $k=k_0$, and $J_1<0$ for $k \in (k_0,1)$. 
	\item[$2)$] If $k \in (0,1)$, $u_1= u^1_z (k)$, $x=1$, then $J_1=0$.
\end{enumerate}
\end{lem}
\begin{proof}
From equalities \eq{d4}, \eq{d0+d2}, \eq{d0} we get for $u_1=u^1_z(k)$, i.~e., if $f_{zu}(u_1, k) =0$:
\begin{eqnarray}
&&d_4=0, \nonumber\\
&&d_2=-d_0, \nonumber\\
&&J_1=d_0(1-x)=a_1 \sin u^1_z(k)(1-x). \label{J1u1u1z}
\end{eqnarray}
It is shown in~\cite{max3} that the function $\sin u^1_z(k) = \sn p^1_z(k)$ is negative for $k \in (0, k_0)$, is equal to zero if $k=k_0$, and is positive for $k \in(k_0, 1)$. Thus for $x \in [ 0, 1)$ there holds the equality $\sgn J_1 = - \sgn a_1 \cdot \sgn (2E(k)-K(k))$. To finish the proof of item 1) of the lemma we use Lemma~\ref{lem:a1u1z}: for $u= u^1_z(k)$, $k \neq k_0$ the function $a_1$ is negative. Also, for $k=k_0$, $u = u^1_z(k_0 =\pi)$ we have $J_1= 2E(k_0)- K(k_0)= 0$.

Item 2) of the lemma follows from~\eq{J1u1u1z}.
\end{proof}

\subsection{Global bounds of conjugate time in the subdomain $C_1$}
We prove estimate~\eq{tconjmax} and get the upper bound for the first conjugate time in this subsection.

\begin{thm} \label{th:tconjmaxC1}
If $\lambda \in C_1$, then $\tconj(\lambda) \geq \tmax(\lambda)$.
\end{thm}
\begin{proof}
Let $\lambda = (\varphi, k, \alpha=1) \in C_1$.

1) 
Suppose $k \in (0, k_0)$. It is required to show that $\tconj(\lambda) \geq 4K(k)$.

1.1) 
Let $\sn(\varphi, k) \neq 0$. Consider the family of extremal trajectories
	 \begin{eqnarray*}
	 && q^s(t) = \Exp(\lambda^s, t), \quad t \in [0, t_1^s], \quad s \in (0, k_0), \\
	 && \lambda ^s = \left(\varphi^s, k^s, \alpha =1 \right) \in  C_1,\\
	 && k^s= s, \quad \varphi^s = F( \am (\varphi, k), s), \quad t_1^s = 4K(s).
	 \end{eqnarray*}
	 For any trajectory from this family 
	 \begin{eqnarray*}
	 x^s && =  \sin^2 u_2^s = \sn^2 \tau^s = \sn^2\left(\varphi^s+ \frac{t_1^s}{2}\right)= \sn^2 \left(\varphi^s+ 2K (k^s), k^s\right)= \sn^2 \left(\varphi^s, k^s \right)=\\
	     && = \sin^2 \left( \am \left(\varphi^s, k^s\right)\right)=\sin^2 \left(\am (\varphi,k)\right)=\sn^2 (\varphi, k)\neq 0,\\
	     u_1^s  && = \am \left(\frac{t_1^s}{2}, k^s\right) = \am \left(2K(s), s\right)=\pi,  
	 \end{eqnarray*}
	 therefore from Lemma~\ref{lem:u1pi} $J_1(u_1^s, x^s, k^s)<0$. Namely the endpoint of a trajectory $q^s(t_1^s)$, $s \in (0, k_0)$, is not a conjugate point.
	 According to Corollary~\ref{cor:k=0t},  there exists $\tilde{k} \in (0, k)$, s.~t. the trajectory $q^{\tilde{k}}(t)$, $t \in (0, t_1^{\tilde{k}}]$ does not contain conjugate points.
	 We apply Theorem~\ref{th:conj_hom} to the family of the trajectories $q^s(t)$, $s\in[\tilde{k}, k]$ and see that the trajectory $q^k(t)$, $t \in (0, t_1^k]$, has no conjugate points, i.~e., $\tconj (\lambda) > 4K(k)$.

1.2)
Let $\sn(\varphi, k)=0$.
	 Consider the family of trajectories
	  \begin{eqnarray*}
	 && q^s = \Exp(\lambda^s, t), \quad t \in [0, t_1^s], \quad s \in (0, \eps), \\
	 && \lambda ^s = \left(\varphi^s, k^s, \alpha =1 \right) \in  C_1,\\
	 && k^s= s, \quad \varphi^s= \varphi+s, \quad t_1^s = 4K(s),	 
	 \end{eqnarray*}
	 where $\eps >0$ is a sufficiently small number, s.~t. $\sn(\varphi + s, k) \neq 0$ for $s \in (0, \eps]$. For the trajectories of this family we have
	 \begin{eqnarray*}
	 &&x^s = \sn \left(\varphi^s + \frac{t_1^s}{2}\right) = \sn^2 (\varphi+ s, k) \neq 0, \quad s \in (0, \eps],\\
	 &&u_1^s=\pi,
	  \end{eqnarray*}
	  therefore according to item 1.1) of this theorem, trajectories $q^s(t)$, $t\in (0, t_1^s]$, $s\in(0, \eps]$, have no conjugate points.
Take any $t_1 \in (0, 4K(k))$. Since conjugate times are isolated from each other, it follows that there exists $t_2 \in (t_1, 4K(k))$ that is not a conjugate time along the trajectory $q^0(t)$. Thus the instant $t_2$ is not a conjugate time for all trajectories of the family $q^s(t)$, $s \in [0, \eps]$.
Using Theorem~\ref{th:conj_hom}, we see that the trajectory $q^0(t)$, $t \in (0, t_2]$ has no conjugate points. Therefore the instant $t_1$ is not a conjugate time. Since $t_1 \in (0, 4K(k))$, we obtain the required inequality $\tconj(\lambda)\geq 4K(k)$. Note that the equality is attained in this case: from Lemma~\ref{lem:u1pi} it follows that $J_1(u_1, x, k)=0$ therefore $\tconj(\lambda)=4K(k)$.

2) Suppose $k=k_0$. Take any $t_1 \in (0, 4K(k_0))$ and any $t_2 \in (t_1, 4K(k_0))$, which is not a conjugate time for the trajectory $\Exp(\lambda,t)$. Applying Theorem~\ref{th:conj_hom} to the family 
 \begin{eqnarray*}
	 && q^s = \Exp(\lambda^s, t), \quad t \in [0, t_1^s], \quad s \in (-\eps, 0), \\
	 && \lambda ^s = \left(\varphi^s, k^s, \alpha =1 \right) \in  C_1,\\
	 && k^s= k+ s, \quad \varphi^s= \varphi, \quad t_1^s = t_2,	 
	 \end{eqnarray*} 
where $\eps>0$ is sufficiently small as in item~1.2) of this proof, we see that $\tconj(\lambda)\geq 4K(k_0)$. According to Lemma~\ref{lem:u1pi}, $\tconj(\lambda)\geq 4K(k_0)$.

3) Suppose $k \in(k_0, 1)$.
We claim that for any $x \in [0,1]$ the set $\left\{(u_1, k)| J_1(u_1, x, k)=0\right\}$ is contained between the curves $u_1 =\pi$ and $u_1 = u^1_z(k)$ in a neighborhood of $(u_1,k)=(\pi, k_0)$; it can easily be checked that these curves are smooth and meet in the point $(u_1,k)= (\pi, k_0)$ at the right angle. We have at this point the expansion: 
\begin{eqnarray*}
J_1(u_1, x, k)= &&-4 E^2(k_0)(\pi- u_1)+ x\left[- \frac{8E^3(k_0)}{k_0(1-k^2_0)}(k-k_0)+4E^2(k_0)(\pi - u_1)\right]+\\
&&+ x^2\frac{8k_0}{1-k^2_0} E^3(k_0)(k-k_0)+O((k-k_0)^2+(\pi-u_1)^2).
\end{eqnarray*}
Thus we get:
\begin{eqnarray*}
&&x\neq 1 \then 
\restr{\frac{\partial J_1}{\partial u_1}}{u_1=\pi, k=k_0}
= 4E^2(k_0)(x-1)\neq 0,\\
&&x = 1 \then 
\restr{\frac{\partial J_1}{\partial k}}{u_1=\pi, k=k_0}= - \frac{8E^3(k_0)}{k_0}\neq 0.
\end{eqnarray*}
Therefore for any $x \in [0,1]$ the equation $J_1(u_1, x, k)=0$ defines a smooth curve in a neighborhood of the point $(u_1, k)=(\pi, k_0)$. From Lemmas~\ref{lem:u1pi} and \ref{lem:u1u1maxC1} it follows that for any $x \in [0, 1]$, $k \in (0, 1)$ the function $J_1(u_1, x, k)$ equals zero on the interval $u_1 \in [\pi, u^1_z(k)]$. Therefore the curve $\left\{J_1=0\right\}$ is contained between the curves $\left\{u_1 = \pi\right\}$ and $\left\{u_1 = u^1_z(k) \right\}$ near the point $(u_1, k)=(\pi, k_0)$. Hence for any $x \in [0, 1]$ there exists a neighborhood of the point $(u_1, k)=(\pi, k_0)$, which satisfies the inequality $J_1(u_1, x, k)\neq 0$ for $u_1< \min(\pi, u^1_z(k))=u^1_{\MAX}(k)$.

3.1) Let $\sn^2(\varphi + p^1_z(k), k)\neq 1$. For $x= \sn^2\left(\varphi + \frac{\tmax}{2}\right)= \sn^2(\varphi +p^1_z(k))$, in a neighborhood $O$ of the point $(u_1, k)=(\pi, k_0)$ the function $J_1(u_1, x, k)$ does not vanish for $k >k_0$, $u_1< u^1_z(k)$. Applying Theorem~\ref{th:conj_hom} to the family of trajectories
\begin{eqnarray*}
	 && q^s(t) = \Exp(\lambda^s, t), \quad t \in [0, t_1^s], \quad s \in [k_0, \tilde{k}], \\
	 && \lambda ^s = \left(\varphi^s, k^s, \alpha =1 \right) \in  C_1,\\
	 && k^s= s, \quad \varphi^s= F(\am(\varphi + p^1_z(k), k), s), \\ 
	 &&t_1^s = 2 F(u^1_z(k),s),	 
	 \end{eqnarray*}
we see that the trajectory $\tilde{q}(t)=q^{\tilde{k}}(t)= \Exp(\tilde{\lambda}, t) $, $\tilde{\lambda}= \lambda^{\tilde{k}}=(\tilde{\varphi}, \tilde{k}, \alpha =1)$, $t \in [0, \tilde{t}]$, $\tilde{t_1}= t^{\tilde{k}}_1$, has no conjugate points.

Finally, applying Theorem~\ref{th:conj_hom} to the family 
\begin{eqnarray*}
	 && q^s(t) = \Exp(\lambda^s, t), \quad t \in [0, t_1^s], \quad s \in [\tilde{k},k], \\
	 && \lambda ^s = \left(\varphi^s, k^s, \alpha =1 \right) \in  C_1,\\
	 && k^s= s, \quad \varphi^s= F(\am(\varphi + p^1_z(k), k), s), \\ 
	 &&t_1^s = 2p^1_z(s),	 
	 \end{eqnarray*}
we see that the trajectory $q^k(t)= \Exp (\lambda, t)$, $t \in [0, \tmax(\lambda)]$ has no conjugate points, Q.~E.~D.

3.2) In the case $\sn^2(\varphi + p^1_z(k), k)=1$, the proof of $\tconj(\lambda)\geq 2p^1_z(k)$ is obtained as in item~1.2).

The theorem is completely proved. 	 	 
\end{proof}

\begin{remark}
The lower bound in estimate of conjugate time~\eq{tconjmax} is attained. From Lemmas~$\ref{lem:u1pi}$, $\ref{lem:u1u1maxC1}$ and Theorem~$\ref{th:tconjmaxC1}$ we get:
\begin{eqnarray*}
&&k \in (0, k_0), \quad \sin \varphi =0, \quad \alpha=1 
\then \tconj(\lambda)=4K(k)= \tmax(\lambda),\\
&&k=k_0, \quad \alpha =1 
\then \tconj(\lambda)=4K(k_0)= \tmax(\lambda),\\
&&k \in (k_0, 1), \ \sin^2(\varphi + p^1_z(k))=1, \ \alpha=1
\then \tconj(\lambda)= 2p^1_z(k)=\tmax(\lambda).
\end{eqnarray*}
\end{remark}

In addition to the lower bound from Theorem~\ref{th:tconjmaxC1} we get the upper bound of the first conjugate time in terms of the second Maxwell time
\begin{eqnarray*}
t^2_{\MAX}(\lambda) = \max(2p^1_z, 4K)/\sigma, \quad \lambda \in C_1.
\end{eqnarray*}
For $\alpha=1$ we obtain:
\begin{eqnarray*}
&&k\in (0, k_0) \then t^2_{\MAX}(\lambda)= 2p^1_z(k),\\
&&k=k_0 \then t^2_{\MAX}(\lambda)= 2p^1_z(k_0)=4K(k_0),\\
&&k \in (k_0, 1) \then t^2_{\MAX}(\lambda)= 4K(k).
\end{eqnarray*}

\begin{proposition}
If $\lambda \in C_1$, then $\tconj(\lambda) \leq \tmaxd (\lambda)$. 
\end{proposition}
\begin{proof}
Let $\lambda=(\varphi, k, \alpha)\in C_1$. Also, in the proof of the lower bound of conjugate time we can assume that $\alpha=1$.

If $k \in (0, k_0)$, then for any $x \in [0, 1]$ we have $J_1(\pi, x, k)\leq 0$ (see Lemma~\ref{lem:u1pi}) and $J_1 (u^1_z(k), x, k)\geq 0$  (see Lemma~\ref{lem:u1u1maxC1}), i.~e., the function $J_1 (u_1, x, k)$ changes sign in the segment $u_1 \in [\pi, u^1_z(k)]$. Therefore the corresponding segment $t \in [4K(k), 2p^1_z(k)]= [\tmax(\lambda), \tmaxd (\lambda)]$ contains the first conjugate time.

If $k=k_0$, then for any $x \in [0, 1]$ we get $J_1(\pi, x, k_0)=0$ (see Lemma~\ref{lem:u1pi}) thus $\tconj (\lambda)= 4K(k_0)= \tmax (\lambda)= \tmaxd (\lambda)$.

Finally, if $k > k_0$, then for any $x \in [0, 1]$ $J(u^1_z(k), x, k)\leq 0$ (see Lemma~\ref{lem:u1u1maxC1}), $J_1(\pi, x, k)\geq 0$ (see Lemma~\ref{lem:u1pi}), therefore $\tconj (\lambda) \in [2p^1_z(k), 4K(k)]= [\tmax (\lambda), \tmaxd (\lambda)]$. 
\end{proof}

\begin{remark}
One should not think that the segment $[\tmax (\lambda), \tmaxd (\lambda)]$ contains exactly one conjugate time. Computational experiments in the system Mathematica show that for $\varphi = 0$ and $k \in (0,999; 1)$ this segment contains two conjugate times.
\end{remark}
%\end{proof}

\section{Estimate of conjugate time for $\lambda \in  C_2$}
\label{sec:C2}
The aim of this section is to prove estimate~\eq{tconjmax} in the domain $C_2$ for $\alpha=1$: $\tconj (\lambda)\geq 2Kk$, $\lambda \in C_2$.
Using parameterization of extremal trajectories~\cite{engel} for $\lambda = (\varphi, k, \alpha)\in C_2$, as well as in the domain $C_1$ we get the expression of the Jacobian $J= \frac{\partial(x, y, z, v)}{\partial(t, \varphi, k, \alpha)}$ for $\alpha =1$:
\begin{eqnarray*}
&&J=R\cdot J_1,\\
&&R= - \frac{32}{k(1-k^2)(1- k^2 \sin^2 u_1 \sin^2 u_2)^2}\neq  0, \\
&&J_1 = d_0 + d_2 \sin^2 u_2 +d_4 \sin^4 u_2,\\
&&u_1= \am(p, k), \quad u_2 = \am(\tau, k), \\
&&p= \frac{t}{2k}, \quad \tau=\frac{2\varphi+t}{2k},
\end{eqnarray*}
\begin{eqnarray*}
d_0 =&&\frac{1}{4}\sin u_1 \cos u_1 (4E^3(u_1) \sin (2u_1)- 4 (1 - k^2) \cos (2 u_1) \sqrt{1 - k^2 \sin ^2 u_1} F^2 (u_1)+\\
&&+(8- 8k^2 + k^4 +k^2 (2-k^2) \cos (2 u_1)) \sin(2u_1) F(u_1)+2 (2 -3k^2 + k^4) \sin (2 u_1)F^3(u_1)+ \\
&&+2 k^2 \sqrt{1-k^2 \sin^2 u_1} \sin^2 (2 u_1)+2E^2 (u_1)(6 \cos (2 u_1)\sqrt{1-k^2 \sin^2 u_1} - (2- k^2)\sin (2u_1) \times \\
&&\times F(u_1))- E(u_1)(4(2-k^2)\cos(2u_1)\sqrt{1-k^2 \sin^2 u_1}F(u_1)+2(4-2k^2+ \\
&&+3k^2 \cos(2u_1))\sin 2u_1+4(1-k^2)\sin(2u_1)F^2(u_1),
\end{eqnarray*}
\begin{eqnarray*}
d_2=&& -2k^2 (1-k^2)\cos u_1 \sin^3 u_1 \sqrt{1-k^2 \sin^2 u_1}F^2(u_1)-2k^4 \cos^3 u_1 \sin^3 u_1 \sqrt{1-k^2 \sin ^2 u_1}-\\
&&-2(1 -k^2 \sin^4 u_1)E^3 (u_1)-(2 - 3k^2+ k^4)(1-k^2 \sin^4 u_1)F^3(u_1)+ \\
&&+E^2(u_1) (6 k^2 \cos u_1 \sin^3 u_1 \sqrt{1- k^2 \sin^2 u_1}+(2-k^2)(1-k^2 \sin^4 u_1)F(u_1))+ \\
&&+E(u_1) (k^2 \cos^2 u_1 \sin^2 u_1 (4-3k^2+3k^2 \cos(2u_1))-\\
&&-2k^2(2-k^2)\cos u_1\sin^3 u_1 \sqrt{1 - k^2 \sin^2 u_1} F(u_1)+2(1 -k^2)(1-k^2\sin^4 u_1)F^2(u_1))-\\
&&-\frac{1}{8}k^2(8-8k^2+k^4+k^2(2-k^2)\cos(2u_1))\sin^2(2u_1)F(u_1),
\end{eqnarray*}
\begin{eqnarray*}
d_4=&&2(1 - k^2 \sin^2 u_1) E^3(u_1)-E^2(u_1)(3k^2\cos u_1 \sin u_1\sqrt{1-k^2\sin^2u_1}+(2-k^2)(1-\\
&&-k^2\sin^2 u_1)F(u_1)+\frac{1}{4}(1-k^2)F^2(u_1)(2(2-k^2)(2-k^2+k^2\cos (2u_1))F(u_1)+\\
&&+2 k^2 \sqrt{1-k^2\sin^2 u_1} \sin 2u_1)+\frac{1}{4}E(u_1)(4k^4\cos^2 u_1 \sin^2 u_1 - \\ 
&&-8 (1-k^2)(1-k^2 \sin^2 u_1)F^2(u_1)+2k^2(2-k^2)\sin(2u_1)\sqrt{1-k^2\sin^2 u_1}F(u_1)).
\end{eqnarray*}

\subsection{Conjugate time as $k \rightarrow 0 $}
Asymptotics of the function $J_1$ as $k \rightarrow 0$ has the form:
\begin{align}
&J_1(u_1, x, k) = \frac{k^8}{1024} J_1^0(u_1, x) + o(k^8), \qquad x = \sin^2 u_2, 
\label{J10C2}\\
&J_1^0 (u_1, x) = d_0^0(u_1) + d_2^0 (u_1) x + d_4^0 (u_1) x^2, \nonumber\\
&d_0^0 (u_1) = \frac{1}{8} \cos u_1 \sin u_1 \big((-48 u_1^2 - 3 ) \cos (2 u_1) + 3 \cos (6 u_1)+ (42 u_1 - 64 u_1^3) \sin (2 u_1) + 2 u_1 \sin (6 u_1) \big), \nonumber\\
&d_2^0 (u_1) = - d_4^0 (u_1) = -(\sin (4 u_1) - 4 u_1) (4 u_1^2 + \sin (4 u_1) u_1 + \cos (4 u_1) - 1). \nonumber
\end{align}
First we prove several auxiliary lemmas which give an estimate of the functions $d^0_i$.
\begin{lem}\label{lem:l5}
The function $f_1(u_1) = 8 u_1 + 4 u_1 \cos(4 u_1) - 3 \sin u_1 > 0$ on the interval $u_1 \in (0, \frac{\pi}{2})$.
\end{lem} 

\begin{proof}
We show that the function $g(u_1) = 2+\cos(4 u_1)$ is a comparison function for $f_1(u_1)$ on the interval $u_1 \in (0, \frac{\pi}{2})$.
The inequality $f_1(u_1) \not \equiv 0$ follows from the expansion $f_1(u_1) = \frac{256}{15} u^5 + o(u^5)$.
Note that $g(u_1)>0$ for any $u_1$.
Finally we get the equalities 
$\ds\left(\frac{f_1(u_1)}{g(u_1)}\right)' = \frac{16 \sin^4 (2 u_1)}{(2+\cos(4 u_1))^2}$ and $\ds\frac{f_1(u_1)}{g(u_1)}=\frac{256}{45}u_1^5+o(u_1^5)$.

So $g(u_1)$ is a comparison function for $f_1(u_1)$ thus it follows from Lemma~\ref{lem:l1} that $f_1(u_1)>0$ for $u_1 \in (0, \frac{\pi}{2})$.
\end{proof}

\begin{lem}\label{lem:2.6}
The function $f_2(u_1) = -1 + 4 u_1^2 + \cos(4 u_1) + u_1 \sin (4 u_1) > 0$ on the interval $u_1 \in (0, \frac{\pi}{2})$.
\end{lem} 

\begin{proof}
Let us show that the function $g(u_1) = 4 u_1 + \sin(4 u_1)$ is a comparison function for $f_2(u_1)$ on the interval $u_1 \in (0, \frac{\pi}{2})$.
The inequality $f_2(u_1) \not \equiv 0$ follows from the expansion $f_2(u_1) = \frac{128}{45} u^6 + o(u^6)$.
If $u_1 > 0$, then $u_1+\sin u_1 > 0$, therefore $g(u_1)>0$ for $u_1 \in (0, \frac{\pi}{2})$.
Finally we have the equalities
$\ds\left(\frac{f_2(u_1)}{g(u_1)}\right)' = \frac{(-4 u_1 + \sin (4 u_1))^2}{(4 u_1 + \sin (4 u_1))^2}$ and 
$\ds\frac{f_2(u_1)}{g(u_1)}=\frac{16}{45}u_1^5+o(u_1^5)$.
So $g(u_1)$ is a comparison function for $f_2(u_1)$ thus it follows from Lemma~\ref{lem:l1} that $f_2(u_1)>0$ for $u_1 \in (0, \frac{\pi}{2})$.
\end{proof}

\begin{lem}\label{lem:2.7}
If $u_1 \in (0, \frac{\pi}{2})$, then the function $d_0^0(u_1) = \frac{1}{8}\cos u_1 \sin u_1 \big( (-48 u_1^2-3)\cos (2 u_1)+(42 u_1 - \\
- 64 u_1^3)\sin(2 u_1)+3 \cos(6 u_1)+2 u_1 \sin(6 u_1) \big)> 0$.
\end{lem} 

\begin{proof}
We now prove that the function $g(u_1) = 4 u_1 + \sin(4 u_1)$ is a comparison function for $d_0^0(u_1)$ on the interval $u_1 \in (0, \frac{\pi}{2})$.
The inequality $d_0^0(u_1) \not \equiv 0$ follows from the expansion $d_0^0(u_1) = \frac{4096}{4725} u^{11} + o(u^{11})$.
If $u_1 \in (0, \frac{\pi}{2})$, then $g(u_1) > 0$.
In the equation
$$\left(\frac{d_0^0(u_1)}{g(u_1)}\right)' = \frac{1}{8 u_1^3 \cos^2 u_1 \sin^2 u_1} f_1(u_1) f_2(u_1),
$$ we note that $f_1 (u_1) = 8 u_1 + 4 u_1 \cos (4 u_1) - 3 \sin(4 u_1) > 0$ on the interval $u_1 \in (0, \frac{\pi}{2})$ (see Lemma~\ref{lem:l5}) and $f_2(u_1) = -1+4 u_1^2 + \cos(4 u_1)+u_1 \sin (4 u_1)>0$ on the interval $u_1 \in (0, \frac{\pi}{2})$ (see Lemma~\ref{lem:2.6}). Meanwhile $\frac{d_0^0(u_1)}{g(u_1)}=\frac{4096}{4725}u_1^7+o(u_1^7)$.
So $g(u_1)$ is a comparison function for $d_0^0(u_1)$, therefore it follows from Lemma~\ref{lem:l1} that $ d_0^0(u_1)>0$ for $u_1 \in (0, \frac{\pi}{2})$.
\end{proof}

Now we estimate the function $J^0_1$.

\begin{lem} \label{lem:J10<0C2}
For any $u_1 \in \left(0, \frac{\pi}{2}\right)$, $x \in [0,1]$ the inequality $J^0_1(u_1, x)>0$ holds.
\end{lem}
\begin{proof}
It follows from Lemma~\ref{lem:2.7} that $J^0_1(u_1, 0) = J^0_1(u_1, 1)= d^0_0(u_1)>0$ for all $u_1 \in \left(0, \frac{\pi}{2}\right)$. Further, it follows from Lemma~\ref{lem:2.6} that $d^0_4(u_1)= \frac{1}{2}(\sin(4u_1)-4u_1)f_2(u_1)<0$ for $u_2 \in (0, \frac{\pi}{2})$. Therefore the statement of this lemma follows from Lemma~\ref{lem:l4}. 
\end{proof}

\begin{proposition} \label{prop:k=0uC2}
There exists $\bar{k} \in (0, 1)$, s.~t. for any $k \in (0, \bar{k})$, $u_1 \in \left(0, \frac{\pi}{2}\right)$, $x \in [0, 1]$, the inequality $J_1(u_1, x, k)>0$ holds.
\end {proposition}
\begin{proof}
This proposition is proved in exactly the same way as Proposition~\ref{prop:k=0u}, with the use of Lemma~\ref{lem:J10<0C2}, expansions~\eq{J10C2} and the following expansions:
\begin{eqnarray*}
J_1=  \frac{4}{4725}k^8u^{11}_1++o(k^8 u^{11}_1)+ \frac{4}{135}k^8 u^9_1 x + o(k^8 u^9_1 x)-\frac{4}{135}k^8 u^9_1 x^2 + o(k^8 u^9_1 x^2), \quad k^2+u^2_1\rightarrow 0,
\end{eqnarray*}
\begin{eqnarray*}
J_1&= \frac{2\pi^2}{8192}k^8\left(\frac{\pi}{2}-u_1\right)+ o \left(k^8\left(\frac{\pi}{2}-u_1\right)\right)+\frac{\pi^3}{512} k^8 x&+o(k^8 x)-\\
&&-\frac{\pi^3}{512}k^8 x^2+o(k^8 x^2), \quad k^2+\left(\frac{\pi}{2}-u_1\right)^2\rightarrow 0.
\end{eqnarray*}
\end{proof}

From Proposition~\ref{prop:k=0uC2} we get the following statement in the variables $(t, \varphi, k)$.

\begin{cor}
There exists $\bar{k} \in (0, 1)$, s.~t. for any $k \in (0, \bar{k})$, $\varphi \in \R$, the trajectory $\Exp (\lambda, t)$, $\lambda = (\varphi, k, \alpha) \in C_2$, $t \in (0, \tmax(\lambda))$, does not contain conjugate points.
\end{cor}

\subsection{Conjugate time for $t = \tmax$}
The instant of time $\tmax(\lambda)=2Kk$ corresponds to the value of the variable $u_1=\frac{\pi}{2}$. 
We have
\begin{eqnarray}
&&J_1\left(\frac{\pi}{2}, x, k\right)=d^{\frac{\pi}{2}}_2 x + d^{\frac{\pi}{2}}_4 x^2, \label{J1pi/2}\\
&&d^{\frac{\pi}{2}}_4 = -d^{\frac{\pi}{2}}_2= \sqrt{1-k^2} \, g_z(K, k)\, f_4(k), \label{d4pi/2}\\
&&g_z(p, k)=((k^2-2)p+2\E(p))\dn p-k^2 \sn p \cn p, \label{gz}\\
&&f_4(k)= E^2(k)-(1-k^2)K(k). \label{f4}
\end{eqnarray}
In the paper~\cite{max3} it was proved that $g_z(p,k)<0$ for any $p>0$, $k \in (0,1)$; therefore $g_z(K,k)<0$.
\begin{lem}\label{lem:l9}
The function $f_4(k) = E^2(k) + (k^2-1) K(k) > 0$ on the interval $k \in (0, 1)$.
\end{lem} 

\begin{proof}
We show that the function $g(k) = 1-k^2$ is a comparison function for $f_4(k)$ on the interval $k \in (0, 1)$.
The inequality $f_4(k) \not \equiv 0$ follows from the expansion $f_4(k) = \frac{\pi^2}{32} k^4 + o(k^4)$.
Note that $g(k)>0$ for $k \in (0,1)$.
Finally we have the equalities
$$\ds\left(\frac{f_4(k)}{g(k)}\right)' = \frac{2 (E(k)+(k^2-1)K(k))^2}{k (k^2-1)^2}$$ and $\ds\frac{f_4(k)}{g(k)}=\frac{\pi^2}{32} k^4 + o(k^4)$.
So $g(k)$ is a comparison function for $f_4(k)$ thus it follows from Lemma~\ref{lem:l1} that $f_4(k)>0$ for $k \in (0, 1)$.
\end{proof}
\begin{lem} \label{lem:u1pi/2C2}
\begin{itemize}
\item[$1)$] 
If $k \in (0,1)$, $u_1 = \frac{\pi}{2}$, $x \in (0,1)$, then $J_1 > 0$.
\item[$2)$] 
If $k \in (0,1)$, $u_1 = \frac{\pi}{2}$, $x \in \left\{0,1\right\}$, then $J_1 = 0$.
\end{itemize}
\end{lem}
\begin{proof}
It follows from formula~\eq{J1pi/2}--\eq{f4}, inequality $g_z(K,k)<0$, and Lemma~\ref{lem:l9}.
\end{proof}

\subsection{Global bounds of conjugate time}
\begin{thm} \label{th:tconjmaxC2}
If $\lambda \in C_2$, then $\tconj(\lambda)\geq \tmax (\lambda)$.
\end{thm}
\begin{proof}
This theorem is proved in exactly the same way as Theorem~\ref{th:tconjmaxC1} based on homotopy invariance of index of the second variation (the number of conjugate points), see Theorem~\ref{th:conj_hom}. The last theorem is applied to the family of extremal trajectries
\begin{eqnarray*}
	 && q^s(t) = \Exp(\lambda^s, t), \quad t \in [0, t^s_1], \quad s \in [\tilde{k}, k], \\
	 && \lambda ^s = \left(\varphi^s, k^s, \alpha =1 \right) \in  C_2,\\
	 &&\varphi^s =sF\left(\am \left(\frac{\varphi + t^k_1/2}{k}, k\right), s\right)-t^s_1/2,\\
	 && k^s= s, \quad t^s_1 = 2K(s)s, \quad \tilde{k} \in (0, \bar{k}).	 
	 \end{eqnarray*}
\end{proof}

\begin{remark}
It follows from Lemma~$\ref{lem:u1pi/2C2}$ that the lower bound from Theorem~$\ref{th:tconjmaxC2}$ is attained: if $\varphi = Kkn$, $n \in \Z$, then $\tconj(\lambda)= 2Kk$, $\lambda = (\varphi, k, \alpha=1) \in C_2$.
\end{remark}

\begin{remark}
Using the homotopy invariance of the index of the second variation, we can prove the upper bound of conjugate time:
\begin{eqnarray*}
&&\tconj(\lambda)\leq \tmaxd(\lambda), \quad \lambda \in C_2,\\
&&\tmaxd(\lambda) = 4kK, \quad \lambda \in C_2.
\end{eqnarray*}
Note that the segment $[\tmax(\lambda), \tmaxd (\lambda)]$ contains exactly two conjugate times (with account of multiplicity).
\end{remark}

\section{Estimate of conjugate time for $\lambda \in C_3$}
\label{sec:C3}
\begin{thm} \label{th:tconjmaxC3}
If $\lambda \in C_3$, then the extremal trajectory $\Exp (\lambda, t)$, $t \in (0, +\infty)$, does not contain conjugate points.
\end{thm}
\begin{proof}
Let $\lambda =(\varphi, k =1, \alpha=1) \in C_3$ and $t_1>0$.
We show that the trajectory $\Exp(\lambda, t)$ , $t \in (0, t_1]$, has no conjugate points.
Choose a time $t_2 >t_1$ that is not a conjugate time. There exists $k_1 \in (0,1)$, s.~t. $k_1K(k_1)> 2t_2$. According to Theorem~\ref{th:tconjmaxC2}, all trajectories
\begin{eqnarray*}
	 && q^s(t) = \Exp(\lambda^s, t), \quad t \in (0, t^s_1], \quad s \in [k_1, 1), \\
	 && \lambda ^s = \left(\varphi^s, k^s, \alpha =1 \right) \in  C_2,\\
	 && \varphi^s= \varphi, \quad k^s= s, \quad t^s_1 =\frac{1}{2}K(s)s, 	 
\end{eqnarray*}
have no conjugate points. Applying Theorem~\ref{th:conj_hom} to the family of trajectories $q^s(t)$, $t \in (0, t_2]$,  $s \in [k_1, 1]$, we conclude that the trajectory $q^1(t)= \Exp (\lambda, t)$, $t \in (0, t_2]$, does not contain conjugate points.
\end{proof}

\section{Estimate of conjugate time for $\lambda \in \cup^{7}_{i=4}C_i$}
\label{sec:C47}
If $\lambda \in \cup^{7}_{i=4}C_i$, then conjugate time (and cut time) can be located by projecting the original problem~\eq{pr1}--\eq{pr3} into simpler problems of a lower dimension using the following proposition. 

\begin{proposition}
\label{prop:project}
Let us consider two optimal control problems:
\begin{eqnarray*}
&& \dot{q}^i = f^i(q^i, u),\quad q^i \in M^i, \quad u \in U, \\
&&q^i(0)= q^i_0, \quad q^i(t_1)= q^i_1,\\
&&J= \int^{t_1}_0 \varphi(u) \,dt \rightarrow \min,\\
&& i=1,2.
\end{eqnarray*}
Suppose that there exists a smooth map $G:M^1 \rightarrow M^2$, s.~t. if $q^1(t)$ is the trajectory of the first system corresponding to a control $u(t)$, then $q^2(t)= G(q^1(t))$ is the trajectory of the second system with the same control.

Further assume that $q^1(t)$ and $q^2(t)$ are such trajectories. If $q^2(t)$ is locally (globally) optimal for the second problem, then $q^1(t)$ is locally (globally) optimal for the first problem.
\end{proposition}

\begin{proof}
Assume the converse. Suppose $q^2(t)$ is optimal and $q^1(t)$ is not optimal. Then for the first problem there exists a trajectory $\tilde{q}^1(t)$, s.~t. value of the functional $J$ for this trajectory is less than for $q^1(t)$. So value of $J$ is less on the trajectory $\tilde{q}^2(t)= G (\tilde{q}^1(t))$ than on $q^2(t)$. This contradiction proves the proposition.
\end{proof}

In the case $\lambda \in C_4 \cup C_5 \cup C_7$, the sub-Riemannian problem on the Engel group is projected on the Riemannian problem in the Euclidean plane $\R^2_{x,y}$: 
\begin{eqnarray*}
&&G: \R^4_{x, y, z, v}\rightarrow \R^2_{x,y}, \qquad (x, y, z, v) \mapsto(x, y),\\
&&\dot{x}= u_1, \quad \dot{y}=u_2, \quad (x,y)(0)= (x_0, y_0), \quad (x,y)(t_1)=(x_1, y_1),\\
&&l= \int^{t_1}_0 \sqrt{u^2_1 + u^2_2} \, dt \rightarrow \min.
\end{eqnarray*}
If $\lambda \in C_4 \cup C_5 \cup C_7$, $\Exp (\lambda, t)= (x_t, y_t, z_t, v_t)$, then $(x_t, y_t)$ is a straight line that is globally optimal in the Riemannian problem on $\R^2_{x, y}$ for $t \in [0, +\infty)$.
Therefore $\tcut(\lambda) = \tconj(\lambda)= +\infty = \tmax (\lambda)$ for $\lambda \in C_4\cup C_5 \cup C_7$.

If $\lambda \in C_6$, then the sub-Riemannian problem on the Engel group is projected on the sub-Riemannian problem on the Heisenberg group $\R^3_{x,y,z}$:
\begin{eqnarray*}
&&G:\R^4_{x, y, z, v} \rightarrow \R^3_{x, y, z}, \qquad (x, y, z, v)\mapsto (x, y, z),\\
&&\dot{x}=u_1,\quad \dot{y}=u_2,\quad \dot{z}=-\frac{y}{2}u_1+\frac{x}{2}u_2,\\
&&(x, y, z)(0)=(x_0, y_0, z_0), \quad (x,y, z)(t_1)=(x_1, y_1, z_1),\\
&&l= \int^{t_1}_0 \sqrt{u^2_1 + u^2_2} \,dt \rightarrow \min.
\end{eqnarray*}
For $\lambda = (\theta, c, \a = 0) \in C_6$, $\Exp(\lambda, t)= (x_t, y_t, z_t, v_t)$, the curve $(x_t, y_t, z_t)$ is globally and locally optimal for $t \in [0, \frac{2\pi}{|c|}]$, i.~e., up to the first turn of the circle $\ds(x_t, y_t)=\left(\frac{\cos(ct + \theta)-\cos \theta}{c}, \frac{\sin(ct)- \sin \theta}{c}\right)$. It follows from Proposition~\ref{prop:project} that $\tconj(\lambda)\geq \tcut (\lambda)\geq \frac{2\pi}{|c|}= \tmax(\lambda)$ for $\lambda \in C_6$. By Theorem~\ref{th:tcut_bound}, we have $\tconj(\lambda)\geq \tcut (\lambda)\leq \frac{2\pi}{|c|}= \tmax(\lambda)$ for $\lambda \in C_6$.

\begin{remark}
Passing to the limit $\alpha \rightarrow 0$, $k \rightarrow 0$, it can be shown that for $\lambda =(\theta, c, \alpha) \in C_6$, $\theta= \alpha = 0$,  equality $\tconj (\lambda) = \frac{2\pi}{|c|}= \tmax(\lambda)$ holds. But for $\lambda \in C_6$ this equality does not hold in the general case.
\end{remark}

Finally we summarize the results of this section in the following statement.
\begin{thm}\label{th:tconjmaxC47}
If $\lambda \in C_4 \cup C_5\cup C_7$, then $\tconj(\lambda)=\tcut(\lambda)= +\infty= \tmax (\lambda)$.
If $\lambda \in C_6 $, then $\tconj(\lambda)\geq \tcut(\lambda)=  \tmax (\lambda)$.
\end{thm}

\section{Conclusion}
Theorem~\ref{th:tconjmax} follows from Theorems~\ref{th:tconjmaxC1}, \ref{th:tconjmaxC2}, \ref{th:tconjmaxC3}, \ref{th:tconjmaxC47}.

Using the estimate of cut time obtained by the work~\cite{engel} (Theorem~\ref{th:tcut_bound}) and the estimate of conjugate time proved in this work (Theorem~\ref{th:tconjmax}), we can get the description of global structure of the exponential map in sub-Riemannian problem on the Engel group. So we can reduce this problem to solving the system of algebraic equations. This will be the subject of another paper.

The method for estimating a conjugate time used in this paper was successfully applied earlier to Euler's elastic problem~\cite{el_conj} and sub-Riemannian problem on the group of rototranslations~\cite{cut_sre1}. There is no doubt that this method is also valid for nilpotent sub-Riemannian problem with the growth vector (2,3,5)~\cite{dido_exp, max1, max2, max3}.

The method can be used for other invariant sub-Riemannian problems on Lie groups of low-dimensional integrable in  non-elementary functions.

The first natural step in this direction is investigation of invariant sub-Riemannian problem on 3D Lie groups which are classified by A.A.Agrachev and D.Barilari~\cite{agrachev_barilari}.

\end{document}